\documentclass[11pt]{article}

\newcommand{\reals}{\mbox{$\mathbb R$}}

\newcommand{\nats}{\mbox{$\mathbb N$}}
\newcommand{\ints}{\mbox{$\mathbb Z$}}

\newcommand{\comment}[1]{}

\usepackage{amsmath}  
\usepackage{amssymb}

\def\squarebox#1{\hbox to #1{\hfill\vbox to #1{\vfill}}}
\def\qed{\hspace*{\fill}
        \vbox{\hrule\hbox{\vrule\squarebox{.667em}\vrule}\hrule}\smallskip}
\newenvironment{proof}{\begin{trivlist}
  \item[\hspace{\labelsep}{\em\noindent Proof.~}]
  }{\qed\end{trivlist}}

\newtheorem{lemma}{Lemma}[section]
\newtheorem{theorem}[lemma]{Theorem}
\newtheorem{corollary}[lemma]{Corollary}
\newtheorem{proposition}[lemma]{Proposition}
\newtheorem{claim}[lemma]{Claim}
\newtheorem{observation}[lemma]{Observation}
\newtheorem{definition}[lemma]{Definition}

\def\squareforqed{\hbox{\rlap{$\sqcap$}$\sqcup$}}
\def\qed{\ifmmode\squareforqed\else{\unskip\nobreak\hfil
\penalty50\hskip1em\null\nobreak\hfil\squareforqed
\parfillskip=0pt\finalhyphendemerits=0\endgraf}\fi}

\newlength{\tablength}
\newlength{\spacelength}
\newcommand{\tabstar}{\hspace*{\tablength}}
\newcommand{\spacestar}{\hspace*{\spacelength}}
\def\obeytabs{\catcode`\^^I=\active}
{\obeytabs\global\let^^I=\tabstar}
{\obeyspaces\global\let =\spacestar}
\newenvironment{display}{\begingroup\obeylines\obeyspaces\obeytabs}{\endgroup}
\newenvironment{prog}{\begin{display}\parskip0pt\sf}{\end{display}}

\title{On the number of hypercubic bipartitions of an integer}

\author{
{\sl Geir Agnarsson}
\thanks{Department of Mathematical Sciences,
George Mason University, 
MS 3F2, 
4400 University Drive, 
Fairfax, VA 22030, 
{\tt geir@math.gmu.edu}}} 

\date{}

\begin{document}

\maketitle

\begin{abstract}
We revisit a well-known divide-and-conquer maximin recurrence
$f(n) = \max( \min(n_1,n_2) + f(n_1) + f(n_2))$ where the maximum
is taken over all proper bipartitions $n = n_1+n_2$, and we 
present a new characterization of the pairs $(n_1,n_2)$ summing to $n$
that yield the maximum $f(n) = \min(n_1,n_2) + f(n_1) + f(n_2)$. 
This new characterization allows us, for a given $n\in\nats$,
to determine the number $h(n)$ of these bipartitions that yield the said
maximum $f(n)$. We present recursive formulae for $h(n)$, a generating 
function $h(x)$, and an explicit formula for $h(n)$ in terms of a 
special representation of $n$.

\vspace{3 mm}


\vspace{2 mm}

\end{abstract}

\section{Introduction}
\label{sec:intro}

The purpose of this article is to further contribute to the 
study of the maximum number of edges of 
an induced subgraph on $n$ vertices
of the hypercube $Q_k$. In order to do that, we revisit a well-known 
divide-and-conquer maximin recurrence 
and recap some of its properties. As stated later in the section,
it is easy to see that, more generally, 
the solution to this mentioned recurrence equals the
maximum number of edges an induced subgraph on $n$ 
vertices in a rectangular grid ${\ints}^k$ can have.
These considerations were in part initially inspired by 
the heuristic integer sequence 
$0,1,2,4,5,7,9,12,13,15,17,20,\ldots$~\cite[A007818]{wrong}, 
describing the maximal number of edges joining $n = 1, 2, 3, \ldots$
vertices in the cubic rectangular grid ${\ints}^3$, for which no general
formula nor procedure to compute it is given. -- First we set forth 
our basic terminology and definitions.

\paragraph{Notation and terminology}
The set of integers will be denoted by $\ints$,
the set of natural numbers $\{1,2,3,\ldots\}$ by $\nats$, 
and the set of non-negative integers $\{0,1,2,3,\ldots\}$ by ${\nats}_0$.
The base-two logarithm of a real $x$ will be denoted by $\lg x$.
Unless otherwise stated, all graphs in this article
will be finite, simple and undirected. For a graph $G$, its
set of vertices will be denoted by $V(G)$ and its set
of edges by $E(G)$. Clearly $E(G)\subseteq \binom{V(G)}{2}$
the set of all 2-element subsets of $V(G)$. We will denote
an edge with endvertices $u$ and $v$ by $uv$ instead of the
actual 2-set $\{u,v\}$. The {\em order}
of $G$ is $|V(G)|$ and the {\em size} of $G$ is $|E(G)|$.
By an {\em induced subgraph} $H$ of $G$ we mean a subgraph
$H$ such that $V(H)\subseteq V(G)$ in the usual set theoretic
sense, and such that if $u,v\in V(H)$ and $uv\in E(G)$, then
$uv\in E(H)$. If $U\subseteq V(G)$ then the subgraph of $G$ 
induced by $V$ will be denoted by $G[U]$. For $k\in\nats$ a 
{\em rectangular grid}
${\ints}^k$ in our context is a infinite graph with the point set
${\ints}^k$ as its vertices and where two points
$\tilde{x} = (x_1,\ldots,x_k)$ and $\tilde{y} = (y_1,\ldots,y_k)$
are connected by an edge iff the {\em Manhattan distance}
$d(\tilde{x},\tilde{y}) = \sum_{i=1}^k |x_i - y_i| = 1$.
So, two points are connected iff they only differ in one
coordinate, in which they differ by $\pm 1$. This Manhattan
distance measure is the metric corresponding to the
{\em 1-norm}  $\|\tilde{x}\|_1 = \sum_{i=1}^k|x_i|$ in
the $k$-dimensional Euclidean space ${\reals}^k$. The {\em hypercube} $Q_k$
is then the subgraph of the grid ${\ints}^k$ induced by
the $2^k$ points $\{0,1\}^k$. The vertices of the hypercube $Q_k$
are more commonly viewed as binary strings of length $k$ instead
of actually points in the $k$-dimensional Euclidean space. In that case the
Manhattan distance is called the called the {\em Hamming distance}. 
We will not make a specific distinction between these two slightly 
different presentations  of the hypercube $Q_k$. 

Assume we have as set of $n$ distinct vertices $U = \{u_1,\ldots,u_n\}$
in a rectangular grid and consider the induced graph $G[U]$ with
the maximum number of edges among all induced subgraphs on $n$ vertices
of the rectangular grid. We will assume its dimension $k$ to be large enough
so our considerations will not be hindered by its value in any way. 
Each vertex $u_i$ is represented by a point 
$\tilde{x}_i = (x_{i\/1},\ldots, x_{i\/k})$ in ${\ints}^k$,
so we may assume that the $j$-th coordinates 
$x_{1\/j}, x_{2\/j},\ldots,x_{n\/j}$ are not all identical,
since otherwise no induced edge in $G[U]$ will be lost
by projection $\pi_{\hat{\jmath}} : {\ints}^k \rightarrow {\ints}^{k-1}$ 
where the $j$-th coordinate has been removed. In other words,
we may assume that $|\pi_{\hat{\jmath}}(U)| \geq 2$ for each coordinate
$j$. In particular, there is a proper partition $U = U_1 \cup U_2$
where the first coordinate of each vertex in $U_1$ is less than
the first coordinate of each vertex in $U_2$. Let $|U_1| = n_1$
and $|U_2| = n_2$, so $n = n_1 + n_2$ where $n_1, n_2 \geq 1$.
Assuming that $G[U]$ has the maximum number $g(n)$ of edges such an induced graph in the grid
can have, we can have at most $\min(n_1,n_2)$ edges in $G[U]$ parallel
to the first coordinate axis. Since each edge in $G[U]$ between 
a vertex of $U_1$ and a vertex of $U_2$ must be parallel
to the first coordinate axis, all other edges in $G[U]$ are 
in the disjoint union $E(G[U_1])\cup E(G[U_2])$, which ideally
are maximally connected. Trivially we have $g(1) = 0$ and so we 
have the following.
\begin{observation}
\label{obs:grid}
For $n\in\nats$ let $g(n)$ denote the maximum number of 
edges an induced subgraph on $n$ vertices of ${\ints}^k$
can have. Then $g(1) = 0$ and 
\[
g(n) \leq \max_{\substack{n_1 + n_2 = n \\ 
                          n_1, n_2 \geq 1}} 
               \left( \min(n_1,n_2) + g(n_1) + g(n_2)\right).
\]
\end{observation}
{\sc Remark:} It is apriori not clear that we have equality
in Observation~\ref{obs:grid}, since by insisting that we have the maximum 
number of edges the first coordinate in the $k$-dimensional grid allows, this 
restricts
the structure of both the ``upper'' and the ``lower'' 
induced subgraphs with on fewer vertices.

\paragraph{A recap of well-known results}
We next recap some known relevant properties about 
the function $f$ on the nonnegative
integers given by the following divide-and-conquer recursion.
\begin{equation}
\label{eqn:f}
f(n) \left\{ \begin{array}{ll} 
               0 \mbox{ if }n=1, \\
               \lfloor n/2\rfloor + f(\lfloor n/2\rfloor) + f(\lceil n/2\rceil) \mbox{ if } n>1.
             \end{array}
     \right.
\end{equation}
This function $f$ and its number sequence 
$(f(n))_{n=0}^{\infty} = (0,1,2,4,5,7,9,12,13,15,17,20,\ldots)$
is a well-known sequence as given in~\cite[A000788]{sum-of-1s} where it is presented by a slightly
different recursion. In fact, the behaviour of $f(n)$ is interesting also from analytic point
of view as $f(n) = \frac{n}{2}\lg n + O(n)$. The asymptotic behavior of 
$\frac{n}{2}\lg n - f(n)$ was first studied in detail in~\cite{Bellman-Shapiro}
and it has fractal-like shape. It tends to towards the 
Blancmange function~\cite{Trott}, a continuous function which is nowhere 
differentiable~\cite{Tall} on every interval $[2^i,2^{i+1}]$ between two consecutive
powers of 2. The Blancmange-like graph of $f(n) - \frac{n}{2}\lg n$ (a negative function) 
also appears in~\cite[Fig.~1, p.~256]{Mcllroy-SIAM-binary}.

If $s(n)$ denotes the sum of the digits of $n$ when expressed as a binary number
(or just the number of 1s appearing in the binary expression of $n$),
then clearly $s(n) = s(n-1) + 1$ when $n$ is odd, and $s(n) = s(n/2)$ when
$n$ is even, and therefore
\begin{equation}
\label{eqn:s}
s(n) = \left\{ \begin{array}{ll}
                 s((n-1)/2) + 1 & \mbox{ if $n$ is odd }, \\
                 s(n/2)         & \mbox{ if $n$ is even }.
               \end{array}
       \right.
\end{equation}
Also, when we express all the $n$ integers $0,1,\ldots, n-1$ as 
binary numbers, $\lfloor n/2\rfloor$ of them are odd and $\lceil n/2\rceil$ even.
From this and (\ref{eqn:s}) it is evident that
\[
\sum_{i=0}^{n-1}s(i) = \sum_{l=0}^{\lfloor n/2\rfloor - 1}s(2l+1) + \sum_{l=0}^{\lceil n/2\rceil-1}s(2l)
= \lfloor n/2\rfloor + \sum_{l=0}^{\lfloor n/2\rfloor - 1}s(l) + \sum_{l=0}^{\lceil n/2\rceil-1}s(l),
\]
which is same recursion that $f$ satisfies. Hence, we have the following
as stated in~\cite[A000788]{sum-of-1s} and \cite{Mcllroy-SIAM-binary}.
\begin{observation}
\label{obs:sum1}
For $n\in\nats$ we have $f(n) = \sum_{i=0}^{n-1}s(i)$.
\end{observation}
For $n\in\nats$ the number of digits in the binary expression of $n-1$ is 
$k = \lceil\lg n \rceil$. For each $i\in\{0,1,\ldots,n-1\}$ there is a 
corresponding binary point $\tilde{\beta}_k(i) \in \{0,1\}^k$
from the binary expression of $i$ where the last digit of $i$ is the $k$-th coordinate,
the next to last digit is the $(k-1)$-th coordinate and so forth.
\begin{proposition}
\label{prp:lower}
For $n\in\nats$ and $k = \lceil \lg n\rceil$ the $n$ points  
$\{\tilde{\beta}_k(0),\tilde{\beta}_k(1),\ldots,\tilde{\beta}_k(n-1)\}$ induce a subgraph 
in $Q_k$ with $f(n)$ edges.
\end{proposition}
\begin{proof}
For each $i\in\{1,\ldots,n-1\}$ we note that $\tilde{\beta}_k(i)$ is connected
by and edge to exactly $s(i)$ previous points $\tilde{\beta}_k(0),\ldots,\tilde{\beta}_k(i-1)$,
namely, those $s(i)$ points obtained from $\tilde{\beta}_k(i)$ by replacing each of the 1s
by a 0. By Observation~\ref{obs:sum1} the total number of edges is therefore
$\sum_{i=0}^{n-1}s(i) = f(n)$.
\end{proof}
From Observation~\ref{obs:grid} and Proposition~\ref{prp:lower} we have 
$f(n) \leq g(n)$. Also, by Observation~\ref{obs:sum1} we have 
from~\cite{S.Hart-DM-note}, \cite{Mcllroy-SIAM-binary}, \cite {Greene-Knuth}
and~\cite{Delange-sur-la} 
that $f(n)$ satisfies the well-known divide-and-conquer maximin recurrence
\begin{equation}
\label{eqn:max-f}
f(n) = \max_{\substack{n_1 + n_2 = n \\ 
                          n_1, n_2 \geq 1}} 
               \left( \min(n_1,n_2) + f(n_1) + f(n_2)\right).
\end{equation}
The proof that $f(n)$ defined by (\ref{eqn:f}) satisfies 
(\ref{eqn:max-f}) given in~\cite[pages 22 -- 23]{Greene-Knuth} is 
particularly short and slick.

From (\ref{eqn:max-f}) it is evident that $g(n)\leq f(n)$ and hence we have equality,
namely the following, as stated in~\cite{S.Hart-DM-note}.
\begin{corollary}
\label{cor:max-grid}
The maximum number of edges an induced simple graph on $n$ vertices
of a rectangular grid ${\ints}^k$ can have is 
$f(n) = \sum_{i=0}^{n-1}s(i)$, the combined number of $1$s in 
the binary expression of $0,1,\ldots,n-1$. 
\end{corollary}
{\sc Remark:} 
Note that the heuristic integer sequence~\cite[A007818]{wrong} 
and the sequence 
$(f(n))_{n=0}^{\infty} = 
(0,1,2,4,5,7,9,12,13,15,17,20,\ldots)$~\cite[A000788]{sum-of-1s} 
agree in the first twelve entries, but differ in the entries 
from and including thirteen. 
By Corollary~\ref{cor:max-grid} this means that the maximum
number of edges of an induced graph on $n$ vertices of an
arbitrary rectangular grid ${\ints}^k$ for $n\in\{1,2,\ldots, 12\}$ 
can be realized in ${\ints}^3$. 

The divide-and-conquer maximin recurrence (\ref{eqn:max-f}) 
is the best-known and most studied one, and one of the very
few with an exact solution given by Observation~\ref{obs:sum1}.
This is mainly since it occurs naturally when analysing worst-case
scenarios in sorting algorithms~\cite{Li-Reingold-maximin}
where both asymptotic results and some other
exact solutions to more general divide-and-conquer maximin recurrences
are given. The fact that $f(n)$ defined by (\ref{eqn:f})
also satisfies (\ref{eqn:max-f}) is a consequence of a special
case of the general treatment in~\cite{Li-Reingold-maximin}. 
In~\cite{Wang-tighter} some of the general asymptotic bounds 
from~\cite{Li-Reingold-maximin} are improved further. --  The 
other reason the divide-and-conquer maximin recurrence (\ref{eqn:max-f}) 
has been studied widely is because its solution $f(n)$ appears as
the answer to extremal combinatorial problems as 
in~\cite{S.Hart-DM-note} where the main result is that of 
Corollary~\ref{cor:max-grid}. In earlier articles like~\cite{Harper-wrong}
and~\cite{Bernstein-right} a procedure is given
on how to place the numbers $1,\ldots, 2^k$ on the vertices
of the hypercube $Q_k$ so the sum $\sum|i-j|$ over all neighbors
of $Q_k$ is minimized. Also, $f$ satisfying (\ref{eqn:max-f}) 
appears when studying the number of 1's in binary integers directly, 
as is done in~\cite{Mcllroy-SIAM-binary}. There the main result
is the presentation of tight closed lower and upper bounds 
for $f(n)$, but also a description of which $n_1$ and $n_2$
adding up to $n$ in (\ref{eqn:max-f}) will yield the maximum
of $f(n)$. In the next section we give a new geometric 
characterization of the those pairs of naturals numbers adding up to $n$
that yielding the maximum $f(n)$ in (\ref{eqn:max-f}). 

\section{Hypercubic bipartitions of an integer}
\label{sec:HCBP}

In his section we give a geometric characterization 
of the natural numbers $n_0$ and $n_1$ summing up to
$n$ such that $f(n) = n_1 + f(n_1) + f(n_0)$. We also
give a direct algebraic parametrization of such ordered
pairs $(n_0,n_1)$. We then enumerate them for each
fixed $n\in\nats$ in the following section.
\begin{definition}
\label{def:HCBP}
For $n\geq 2$ and $k = \lceil \lg n\rceil$, a partition
$n = n_0 + n_1$ with $n_0,n_1\geq 1$ and $n_0\geq n_1$ 
is a {\em hypercubic bipartition (HCBP)} if there
is an $i\in\{1,\ldots,k\}$ 
such that the hyperplane 
$x_i = 1/2$ splits the $n$ points 
$\{\tilde{\beta}_k(0),\tilde{\beta}_k(1),\ldots,\tilde{\beta}_k(n-1)\}$ 
into two parts containing $n_0$ points on one side of the hyperplane
and $n_1$ on the other.
\end{definition}
{\sc Remark:} Of course, a partition of $n$ is a HCBP 
iff there is an $i\in\{1,\ldots,k\}$ such that among the $n$ points 
$\{\tilde{\beta}_k(0),\tilde{\beta}_k(1),\ldots,\tilde{\beta}_k(n-1)\}$ 
there are $n_0$ of them with $i$-th coordinate $0$ and $n_1$
of them with $i$-th coordinate $1$.\footnote{this is the
reason for our change in labeling to $(n_0,n_1)$ from 
$(n_1,n_2)$ in (\ref{eqn:max-f}).}

\vspace{3 mm}

Our first objective is to prove the following equivalence.
\begin{theorem}
\label{thm:HCBP=max}
For $n\in\nats$, a partition $n = n_0 + n_1$ with $n_0\geq n_1$
is a HCBP if and only if $f(n) = n_1 + f(n_1) + f(n_0)$.
\end{theorem}
One direction is fairly straightforward; 
by Proposition~\ref{prp:lower} and Corollary~\ref{cor:max-grid}
the $n$ points $\beta_k(0),\ldots,\beta_k(n-1)$ form a maximally
connected subgraph of $Q_k$ with $f(n)$ edges. 
Assume $n = n_0 + n_1$ is a HCBP. Then for
some $i\in\{1,\ldots,k\}$ there 
are $n_0$ points with $i$-th coordinate $0$ and
$n_1$ points with $i$-th coordinate $1$.
Among the $f(n)$ edges precisely $n_1$ of them
are parallel to the $i$-th axis. At most $f(n_0)$
of the remaining edges connect points with $i$-th coordinate
$0$ and at most $f(n_1)$ connect points with $i$-th coordinate
$1$. By maximality of $f(n)$ we therefore have 
$f(n) = n_1 + f(n_1) + f(n_0)$, and hence the following.
\begin{observation}
\label{obs:HCBP->max}
If $n\geq 2$ and $n=n_0+n_1$ is a HCBP 
then $f(n) = n_1 + f(n_1) + f(n_0)$.
\end{observation}
To verify the other direction, we will do so in a number of small steps.
For each of them we attempt to keep
our arguments as elementary as possible.
The first one is obtaining an equivalent algebraic description of a 
HCBP of $n$. 

Consider the $n\times k$ matrix ${\mathcal{B}}_n$ where 
$k = \lceil \lg n\rceil$ whose $i$-th row vector
is the point $\beta_k(i-1)$ for $i=1,\ldots,n$. For each
$i\in\{1,\ldots,k\}$ let $d_i(n)$ denote the
difference between the number of $0$'s and the number of $1$'s 
in the $i$-th column of ${\mathcal{B}}_n$. For convenience we set $d_i(0) = 0$ for each $i$.
We then have 
\begin{equation}
\label{eqn:d}
d_i(n) = 2^{i-1} - | (n\bmod {2^i}) - 2^{i-1}| = 2^{i-1} - | n - 2^i\lfloor n/2^i\rfloor -  2^{i-1}|.
\end{equation}
If the $i$-th column of ${\mathcal{B}}_n$ has 
$n_0$ zeros and $n_1$ ones, then $n_0 - n_1 = d_i(n)$
and $n = n_0 + n_1$ is a HCBP. The converse is clear and hence
we have the following characterization.
\begin{observation}
\label{obs:HCBP=d}
For $n\geq 2$ the partition $n = n_0 + n_1$ 
with $n_0,n_1\geq 1$ and $n_0\geq n_1$ is a HCBP if and only if
$n_0 - n_1 = d_i(n)$ for some $i\in\{1,\ldots,k\}$ 
where $k = \lceil \lg n\rceil$.
\end{observation}
Viewing $i$ as fixed the graph of the map $n\mapsto d_i(n)$ 
has a zigzag like shape where each zig has length/period $2^i$.
The function $d_i$ can easily be extended to all non-negative integers $n$.
\begin{claim}
\label{clm:diff}
$d_i(m) = d_i(n)$ if and only if $m \equiv\pm n\pmod {2^i}$. 
\end{claim}
\begin{proof}
The function $d_i:{\nats}_0\rightarrow{\nats}_0$ is the identity
on $\{0,1,\ldots, 2^{i-1}\}$, is even, and has period $2^i$.
\end{proof}
The following properties are also straightforward from the definition (\ref{eqn:d}).
\begin{claim}
\label{clm:didj}
For $i,j\in\nats$ and $m,n\in{\nats}_0$ we have
\begin{enumerate}
  \item If $i\leq j$, then $d_i(n) \leq d_j(n)$.  
  \item If $i\leq j$ and $d_j(m) = d_i(n)$, then $d_j(m) = d_i(m)$. 
\end{enumerate}
\end{claim}
\begin{claim}
\label{clm:2n}
For $n\in{\nats}_0$ and $i\in\nats$ we have:
\begin{enumerate}
  \item $d_{i+1}(2n) = 2d_i(n)$.
  \item $d_{i+1}(2n+1) = d_i(n+1) + d_i(n)$.
\end{enumerate}
\end{claim}
For our next lemma, assume that $d_j(n+1) = d_i(n)\pm 1$ for some $i,j$.
Since in general $d_{\ell}(n+1) = d_{\ell}(n)\pm 1$ for each $\ell$, then
we have a total eight cases to consider.

If $d_j(n+1) = d_i(n) + 1$ for some $i$ and $j$, and
$d_i(n+1) = d_i(n) - 1$ and $d_j(n+1) = d_j(n) - 1$,
then $d_j(n) = d_i(n) + 2$ and $d_j(n+1) = d_i(n+1) + 2$.
Therefore we have that $i<j$ and both $d_i$ and $d_j$ are decreasing from $n$ to $n+1$.
This can only occur when $i=1$, $j\geq 3$ and $n\equiv -3\pmod {2^j}$.
For these values of $i$ and $j$ we now have that
$d_3(2n+1) = d_2(n+1) + d_2(n)  = 3 = d_j(n+1) + d_i(n)$. 

Similarly, if $d_j(n+1) = d_i(n) - 1$ for some $i$ and $j$, and
$d_i(n+1) = d_i(n) + 1$ and $d_j(n+1) = d_j(n) + 1$,
then $d_i(n+1) - d_j(n+1) = d_i(n) - d_j(n) = 2$.
Hence, $i>j$ and both $d_i$ and $d_j$ are increasing from $n$ to $n+1$.
This can only occur when $j=1$, $i\geq 3$ and $n\equiv 2\pmod {2^i}$.
For these values of $i$ and $j$ we now have that
$d_3(2n+1) = 3 = d_j(n+1) + d_i(n)$. 

In all the remaining cases we obtain either $d_j(n) = d_i(n)$ or 
$d_j(n+1) = d_i(n+1)$, and hence we obtain
by Claim~\ref{clm:2n} the following.
\begin{lemma}
\label{lmm:j=i}
Let $i,j\in\nats$. If $d_j(n+1) = d_i(n)\pm 1$
then $d_j(n+1) + d_i(n) = d_{\ell}(2n+1)$ for
some $\ell\in \{3,i+1,j+1\}$.
\end{lemma}

We now have what we need to complete the more involved part
of the proof of Theorem~\ref{thm:HCBP=max}.
We proceed by induction on $n$. The cases $n\leq 3$ being
trivial, we consider even and odd cases, assume that 
the statement of Theorem~\ref{thm:HCBP=max} holds
for all natural numbers less than $2n$, and show it then
holds for both $2n$ and $2n+1$.

{\sc Case 1:} If $f(2n) = 2n_1 + f(2n_1) + f(2n_0)$ where $n = n_0 + n_1$ and $n_0\geq n_1$,
then directly by (\ref{eqn:f}) we get $f(n) = n_1 + f(n_1) + f(n_0)$. 
By induction hypothesis $n = n_0 + n_1$ is a HCBP and hence by 
Observation~\ref{obs:HCBP=d} $n_0 - n_1 = d_i(n)$ for some $i$.
By Claim~\ref{clm:2n} we have $d_{i+1}(2n) = 2d_i(n)$, and hence
$2n_0 - 2n_1 = d_{i+1}(2n)$ which shows that $2n = 2n_0 + 2n_1$ is a HCBP.

{\sc Case 2:} If $f(2n) = 2n_1 + 1  + f(2n_1 + 1) + f(2n_0 - 1)$ 
where $n = n_0 + n_1$ and $n_0\geq n_1 + 1$, then by the defining
recursion (\ref{eqn:f}) we obtain $2f(n) = 2n_1 + f(n_1) + f(n_1+1) + f(n_0 -1) + f(n_0)$.
By (\ref{eqn:max-f}) we have in general that $f(n-1) \geq n_1 + f(n_1) + f(n_0-1)$
and $f(n+1)\geq n_1+1 +f(n_1+1) + f(n_0)$ and hence 
$2f(n)+1\leq f(n-1) + f(n+1)$ in this case. Since $2f(n) = f(2n) - n$ by 
(\ref{eqn:f}), we obtain by (\ref{eqn:max-f}) the opposite inequality and so we have here
equality in both inequalities, so $2f(n) + 1 = f(n-1) + f(n+1)$.
By Observation~\ref{obs:sum1} this can be rewritten as 
$s(n-1) + 1 = s(n)$ which means that $n$ is odd.
Also, by induction hypothesis
both $n-1 = n_1 + (n_0-1)$ and $n+1 = (n_1+1) + n_0$ are HCBP and
hence $(n_0 - 1) - n_1 = d_i(n-1)$ and $n_0 - (n_1 + 1) = d_j(n+1)$
for some $i$ and $j$ and hence $d_i(n-1) = d_j(n+1)$. 
By Claim~\ref{clm:didj} $d_{\ell}(n-1) = d_{\ell}(n+1)$ for some $\ell\in\{i,j\}$,
and hence by Claim~\ref{clm:diff} either $2$ or $2n$ is divisible by $2^{\ell}$. 
Since $n$ is odd we must have $\ell=1$ and therefore $n_0 = n_1+1$.
As $2n_0 - 1 = 2n_1 + 1$, then $2n = (2n_1 + 1) + (2n_0 - 1)$
is a HCBP.

{\sc Remark:}
We see conversely that if $n = n_0 + n_1$ is a HCBP,
then by Observation~\ref{obs:HCBP=d} $2n_0 = n+d_i(n)$ 
and $2n_1 = n-d_i(n)$ for some $i$. 
Since $d_i(n) \equiv\pm n\pmod {2^i}$ then either
$n_0$ or $n_1$ is divisible by $2^{i-1}$.
Therefore if a bipartition of $2n$  
into two odd parts is a HCBP, then both parts must be equal.

{\sc Case 3:} If $f(2n + 1) = 2n_1 + f(2n_1) + f(2n_0 + 1)$ 
where $n = n_0 + n_1$ and $n_0\geq n_1$, then by (\ref{eqn:f})
we get $f(n) + f(n+1) = 2n_1 + 2f(n_1) + f(n_0) + f(n_0+1)$
By (\ref{eqn:max-f}) we have in general that $f(n) \geq n_1 + f(n_1) + f(n_0)$
and $f(n+1)\geq n_1+f(n_1) + f(n_0+1)$, and hence in this case we
have equality in both these inequalities, so 
$f(n) = n_1 + f(n_1) + f(n_0)$ and $f(n+1) =n_1+f(n_1) + f(n_0+1)$,
which by induction hypothesis are both HCBPs and hence
$n_0-n_1=d_i(n)$ and $n_0+1-n_1 = d_j(n+1)$ for some $i$ and $j$.
By Lemma~\ref{lmm:j=i} then $d_i(n) + d_j(n+1) = d_{\ell}(2n+1)$ 
for some $\ell$, so $2n+1 = 2n_1 + (2n_0+1)$ is a HCBP.

{\sc Case 4:} Finally, if $f(2n + 1) = 2n_1 + 1 + f(2n_1+1) + f(2n_0)$ 
where $n = n_0 + n_1$ and $n_0\geq n_1+1$, then by (\ref{eqn:f})
we get $f(n) + f(n+1) = 2n_1+1 + f(n_1+1) + f(n_1) + 2f(n_0)$
By (\ref{eqn:max-f}) we have in general that $f(n) \geq n_1 + f(n_1) + f(n_0)$
and $f(n+1)\geq n_1 + 1 + f(n_1+1) + f(n_0)$, and hence in this case we
have equality in both these inequalities, so 
$f(n) = n_1 + f(n_1) + f(n_0)$ and $f(n+1) =n_1+1+f(n_1+1) + f(n_0)$,
which by induction hypothesis are both HCBPs and hence
$n_0-n_1=d_i(n)$ and $n_0-(n_1+1) = d_j(n+1)$ for some $i$ and $j$.
By Lemma~\ref{lmm:j=i} then $d_i(n) + d_j(n+1) = d_{\ell}(2n+1)$ 
for some $\ell$, so $2n+1 = 2n_1 + (2n_0+1)$ is a HCBP. -- This completes
the proof of Theorem~\ref{thm:HCBP=max}.

We conclude this section by summarizing the main results from this
section.
\begin{theorem}
\label{thm:summary}
For $n\in\nats$ and $f$ the function defined in (\ref{eqn:f}), TFAE:
\begin{enumerate}
  \item $n = n_0 + n_1$ is a HCBP.
  \item $f(n) = n_1 + f(n_1) + f(n_0)$ and $n_0\geq n_1\geq 1$.
  \item $(n_0,n_1) = \left(\frac{n+d_i(n)}{2},\frac{n-d_i(n)}{2}\right)$
for some $i\in\{1,\ldots,\lceil \lg n\rceil\}$.
\end{enumerate}
\end{theorem}

{\sc Remark:} In~\cite{Mcllroy-SIAM-binary} a different description of 
nonnegative integer pairs $(n_0,n_1)$ such that
$f(n_0+n_1) = \min(n_0,n_1) + f(n_0) + f(n_1)$ holds is given.
Although there the main focus is on the entire set of such
pairs, as suppose to HCBPs of each fixed $n\in\nats$ as
in Theorem~\ref{thm:summary}, a careful reading of~\cite{Mcllroy-SIAM-binary}
shows that the description in Theorem~\ref{thm:summary} part 3
and in~\cite{Mcllroy-SIAM-binary} are equivalent: plotting
$\{(x,y)\in \{0,1,\ldots\} : (x,y) = \left(\frac{n+d_i(n)}{2},\frac{n-d_i(n)}{2}\right)
\mbox{ where }n\in\nats\mbox{ and } 1\leq i\leq \lceil \lg n\rceil\}$
yields the same set as is described in~\cite{Mcllroy-SIAM-binary} minus
the points on the $x$-axis $y=0$.

\section{Enumerations of HCBPs}
\label{sec:enum}

For $n\in\nats$ let $h(n)$ denote the number of HCBPs 
$n = n_0 + n_1$ where $n_0\geq n_1\geq 1$. In this section
we will determine the generating function $h(x) = \sum_{n\geq 0}h(n)x^n$,
present a efficient recursive procedures to compute $h(n)$, and
present a formula for $h(n)$ in terms of a special presentation of 
each fixed $n\in\nats$. For this purpose it will be convenient to add
the trivial partition $n = n + 0$ to the HCBPs of $n$, 
and so $h(n) = c(n) - 1$, where $c(n)$ is the number of 
``non-proper'' HCBPs of $n$ in which $n_0 = n$ and $n_1 = 0$ is allowed.
By definition of $d_i(n)$ from (\ref{eqn:d}) we note that for 
each $i>\lceil \lg n\rceil$ we have $d_i(n) = n$. 
Hence, by Theorem~\ref{thm:summary}
\[
c(n) = \left| \left\{ \left(\frac{n+d_i(n)}{2},\frac{n-d_i(n)}{2}\right) 
: i\in\nats\right\}\right|,
\]
or equivalently, $c(n)$ is the number of distinct values
of $d_i(n)$ for various $i\in\nats$. From a geometric point
of view, let $\Gamma_i = \{(n,d_i(n)) : n\in\nats_0\}$ be the graph of the map 
$d_i : \nats_0\rightarrow\nats_0$ for each $i$ and let 
$\Gamma = \bigcup_{i\geq 1}\Gamma_i\subseteq \nats_0^2$.
If $V_n = \{(n,y) : y\in\ints\}$ denotes
the integral vertical line in ${\nats_0}^2$ at $n$,
then $c(n) = |\Gamma\cap V_n|$. 
To obtain the generating function $c(x) = \sum_{n\geq 0}c(n)x^n$ 
we partition $\Gamma$ into slices parallel to the $x$-axis as
\[
\label{eqn:Gamma-part}
\Gamma = \bigcup_{i\geq 0}(\Gamma\cap\Pi_i),
\]
where $\Pi_0 = \{(x,y)\in\nats_0^2  : 0\leq y \leq 1\}$ and 
$\Pi_i = \{(x,y)\in \nats_0^2 : 2^{i-1} < y \leq 2^i\}$
for each $i\in\nats$. From this we get that
\begin{eqnarray*}
c(n) & = & \left|\left(\bigcup_{i\geq 0}(\Gamma\cap\Pi_i)\right)\cap V_n\right| \\
     & = & \left|\bigcup_{i\geq 0}(\Gamma\cap\Pi_i\cap V_n)\right|\\
     & = & \sum_{i\geq 0} |\Gamma\cap\Pi_i\cap V_n| \\
     & = & \sum_{i\geq 0} c_i(n),
\end{eqnarray*}
where $c_i(n) = |\Gamma\cap\Pi_i\cap V_n|$. Since $d_i(n)\leq n$ for each fixed
$n$ and all $i$, we have by definition of $\Pi_i$ that $\Gamma\cap\Pi_i\cap V_n =\emptyset$
for $i>\lceil \lg n\rceil$. Hence, the last sum in the above display is a finite sum.
Letting $c_i(x) = \sum_{n\geq 0}c_i(n)x^n$ be the generating function corresponding
to $(c_i(n))_{n\geq 0}$, we get $c(x) = \sum_{i\geq 0}c_i(x)$. Note that $c_0(n)=1$
for each $n\geq 0$. For $i\geq 1$ we have the following.
\begin{lemma}
\label{lmm:c_i(n)}
Let $i\geq 1$ and $n\geq 0$.
If $2^{i-1}<d_j(n)\leq 2^i$, then $d_j(n) = d_{i+1}(n)$.
Further $2^{i-1}<d_{i+1}(n)\leq 2^i$ iff $-2^{i-1}<(n\bmod {2^{i+1}})-2^i<2^{i-1}$.
\end{lemma}
\begin{proof}(Sketch.)
By definition, we have $0\leq d_j(n)\leq 2^{j-1}$ for all $n$.
So if $2^{i-1}<d_j(n)\leq 2^i$, then $i+1\leq j$ must hold.
Assuming $2^{i-1}<d_j(n)\leq 2^i$, we consider two cases. 
(i) If $n\bmod {2^j}\in \{0,1,\ldots,2^{j-1}\}$, then
since $i\leq j-1$ we have $d_j(n) = d_{i+1}(n) = n\bmod {2^{i+1}}$.
(ii) If $n\bmod {2^j}\in \{2^{j-1}+1,\ldots,2^j-1\}$ then
$d_j(n) = 2^j - (n\bmod {2^j})$. Since $j-1\leq i$ we further have 
by our assumption that
\[
n\bmod {2^j} \in \{2^j-2^i,\ldots,2^j-2^{i-1}\}\subseteq \{2^j-2^i,\ldots,2^j-1\}
\]
and hence
$n\bmod {2^{i+1}}\in\{2^i,\ldots,2^{i+1}-1\}$ 
as $n\bmod 2^j = n\bmod 2^{i+1} +2^j -2^{i+1}$. Therefore
$d_{i+1}(n) = 2^{i+1} - n\bmod {2^{i+1}} = 2^j - n\bmod {2^j} = d_j(n)$.
The rest follows from the definition of $d_{i+1}(n)$.
\end{proof}
What the above Lemma~\ref{lmm:c_i(n)} states is that
(i) for each $n\geq 0$ we have $c_i(n)=0,1$, and (ii)
$c_i(n) = 1$ iff $|(n\bmod {2^{i+1}})-2^i|<2^{i-1}$.
Letting $\ell = n\bmod {2^{i+1}}$ and writing 
$n = m2^{i+1}+\ell$, we consequently get for $i\geq 1$ that 
\begin{eqnarray*}
c_i(x) & = & \sum_{\stackrel{0\leq n = m2^{i+1}+\ell}{|\ell-2^i|<2^{i-1}}}x^n \\
       & = & \sum_{m\geq 0}\left(\sum_{\ell = 2^i-2^{i-1}+1}^{2^i+2^{i-1}-1}x^{m2^{i+1}+\ell}\right)\\
       & = & \frac{x^{2^{i-1}+1}(1-x^{2^i-1})}{1-x}\sum_{m\geq 0}x^{m2^{i+1}}\\
       & = & \frac{x^{2^{i-1}+1}(1-x^{2^i-1})}{(1-x)(1-x^{2^{i+1}})}.
\end{eqnarray*}
Since $c(x) = \sum_{i\geq 0} c_i(x) = c_0(x) + \sum_{i\geq 1} c_i(x)$, we obtain the
following theorem.
\begin{theorem}
\label{thm:genfunc}
The generating function $c(x) = \sum_{n\geq 0}c(n)x^n$ for 
$(c(n))_{n\geq 0}$ is given by
\[
c(x) = \frac{1}{1-x} + \sum_{i\geq 1}\frac{x^{2^{i-1}+1}(1-x^{2^i-1})}{(1-x)(1-x^{2^{i+1}})}.
\]
Hence, since $h(n) = c(n) - 1$ for each $n\geq 2$, the generating function 
$h(x)=\sum_{n\geq 0}h(n)x^n$ for $(h(n))_{n\geq 0}$, the number of 
HCBPs of $n$, is given by
\[
h(x) = \sum_{i\geq 1}\frac{x^{2^{i-1}+1}(1-x^{2^i-1})}{(1-x)(1-x^{2^{i+1}})}.
\]
\end{theorem}
Let $n\in\nats$ and $k = \lceil \lg n\rceil$. For $i\in\{1,\ldots,k-1\}$ we have
$(2^k-n)\bmod {2^{i+1}} = 2^{i+1}-(n\bmod {2^{i+1}})$ and hence
$|((2^k-n)\bmod {2^{i+1}})-2^i| = |(n\bmod {2^{i+1}})-2^i|$.
From this we see that if $i\in\{1,\ldots,k-1\}$ then $c_i(n) = 1$
iff $c_i(2^k-n)=1$. Additionally for $i=k$ we have
$|(n\bmod {2^{k+1}})-2^k| = 2^k-n < 2^{k-1}\mbox{ and }|((2^k-n)\bmod {2^{k+1}})-2^k| = n > 2^{k-1}$
and so $c_k(n) = 1$ and $c_k(2^k-n) = 0$. As $c_0(n) = 1$ for each $n$ and 
$c_i(n) = 0$ for each $i > k$ we have 
\[
c(n) = \sum_{i=0}^kc_i(n) = \left(\sum_{i=0}^{k-1}c_i(n)\right) + c_k(n) 
=  \left(\sum_{i=0}^{k-1}c_i(2^k-n)\right) + 1 = \left(\sum_{i\geq 0}c_i(2^k-n)\right) + 1 = c(2^k-n) + 1,
\]
which yields a recurrence for the $c(n)$, namely 
\begin{equation}
\label{eqn:recc}
c(0)=c(1)=1 \mbox{ and } c(n) = c(2^{\lceil \lg n\rceil} - n) + 1 \mbox{ for }n>1.
\end{equation}
As the number $h(n)$ of HCBPs of $n$ satisfies $h(n) = c(n)-1$ we have
the following
\begin{corollary}
\label{cor:recurr-h}
For $n\in\nats$ the number $h(n)$ of HCBPs of $n$ satisfies
the following determining recurrence
\begin{eqnarray*}
h(1) & = & 0,\\
h(2) & = & 1, \\
h(n) & = & h(2^{\lceil \lg n\rceil} - n) + 1, \mbox{ for } n > 2.
\end{eqnarray*}
\end{corollary}
{\sc Remark:} Note that the map $n\mapsto 2^{k+1} - n$ is the reflection about the vertical line
$x = k$ in the real Euclidean plane ${\reals}^2$. 
Hence, looking at $\Gamma\subseteq {\nats_0}^2\subseteq {\reals}^2$, 
the map $n\mapsto 2^{\lceil \lg n\rceil} - n$ is a reflection about $x = \lceil \lg n\rceil -1$,
$2$ to the power of which is the largest power of 2 strictly less than $n$. 
From the shape of $\Gamma$ it is clear that we have a bijection
\[
(\Gamma\cup V_n)\setminus \{(n,n)\}\rightarrow \Gamma\cup V_{2^{\lceil \lg n\rceil} - n}
\]
given by $(n,y)\mapsto(2^{\lceil \lg n\rceil} - n,y)$. Hence, the recursion in (\ref{eqn:recc})
(and therefore Corollary~\ref{cor:recurr-h}) is also 
evident from this geometrical perspective.

\vspace{3 mm}

If $n\in\nats$ and $k = \lceil \lg n\rceil \geq 2$, then
$2^{k-1}<n\leq 2^k$ and hence $0\leq 2^k-n<2^{k-1}$. Hence 
both $c(n)$ and $h(n)$ can be computed in at most $\lceil \lg n\rceil - 1$
steps\footnote{i.e.~arithmetic operations}, provided that $n\geq 3$. 
In fact, if $n = a_k = (2^{k+1} + (-1)^k)/3$, then by (\ref{eqn:recc})
$c(n) = c(a_k)= c(a_{k-1})+1$, and so $c(n) = k = \lceil \lg n\rceil$, and 
is obtained in exactly $\lceil \lg n\rceil-1$ steps with the recursion in (\ref{eqn:recc}).
\begin{observation}
\label{obs:lg-1}
For $n\in\nats$ we have
\begin{enumerate}
  \item $n\in\{1,2\}$ implies $c(n) = \lceil \lg n\rceil + 1$ and hence $h(n) = \lceil \lg n\rceil$.
  \item $n\geq 3$ implies $c(n)\leq \lceil \lg n\rceil$ and hence $h(n) \leq \lceil \lg n\rceil-1$,
and equality holds for infinitely many $n\geq 3$.
\end{enumerate}
\end{observation}

Ideally we would like to develop an explicit formula for $c(n)$ and hence 
$h(n)$ in terms of $n$. We will conclude this section by the next best thing; 
a formula in terms of a special representation of $n$, similar to
the one in Observation~\ref{obs:sum1} for $f(n)$.

Consider $n\in\nats$ in its binary representation. Consider
each maximal string of 1s in this representation, except the last
one if it is single (that is, if $n$ is of the form $n = 2^a(4m+1)$.)
Rewrite each of these strings as a difference of two powers of twos;
$2^{\alpha} + 2^{\alpha -1} + \cdots + 2^{\beta} = 2^{\alpha +1} - 2^{\beta}$.
In this way we obtain from the binary representation of $n$  
a representation of $n$ as a finite alternating sum of 
powers of two's 
\begin{equation}
\label{eqn:altsum}
n = 2^{\alpha_1} - 2^{\alpha_2} + \cdots + (-1)^{\ell-1}2^{\alpha_{\ell}},
\end{equation}
where 
\begin{equation}
\label{eqn:cond-alpha}
\alpha_1 > \alpha_2 > \cdots > \alpha_{\ell-1} > \alpha_{\ell} + 1.
\end{equation}
\begin{definition}
\label{obs:ABR}
A representation of $n\in\nats$ as (\ref{eqn:altsum}) where the
exponents satisfy (\ref{eqn:cond-alpha}) is called an
{\em alternating binary representation (ABR)} of $n$.
\end{definition}
An ABR of $n\in\nats$ has the following property.
\begin{lemma}
\label{lmm:lgalt}
If $(\alpha_1,\ldots,\alpha_{\ell})$ determines an ABR of $n\in\nats$,
then 
$\alpha_1 = \lceil \lg n\rceil$ and 
$\alpha_i = \lceil \lg((-1)^{i-1}(n - 2^{\alpha_1} + 2^{\alpha_2} 
+ \cdots + (-1)^{i-1}2^{\alpha_{i-1}}))\rceil$ for each $i\in\{2,\ldots,{\ell}\}$.
In particular, the ABR of $n$ is unique and $(\alpha_2,\ldots,\alpha_{\ell})$ 
determines the ABR of $2^{\alpha_1}-n$.
\end{lemma}
\begin{proof}
Let $n = 2^{\alpha_1} - 2^{\alpha_2} + \cdots + (-1)^{\ell-1}2^{\alpha_{\ell}}$ where
$\alpha_1 > \alpha_2 > \cdots > \alpha_{\ell-1} > \alpha_{\ell} + 1$. We proceed
by induction on $\ell$: the statement is clear for $\ell = 1$ as $n$ is then 
a power of $2$. 

Let $\ell\geq 2$. In this case we have $2^{\alpha_1} > n$.
If $\ell = 2$, then $n = 2^{\alpha_1} - 2^{\alpha_2}$ where $\alpha_1 > \alpha_2+1$
and hence $n > 2^{\alpha_2+1} - 2^{\alpha_2} = 2^{\alpha_2}$. So we
have here that $2^{\alpha_2}<n<2^{\alpha_1}$ and therefore
$\alpha_1 = \lceil \lg n\rceil$. Since $2^{\alpha_1} - n = 2^{\alpha_2}$
the lemma holds in this case.

If $\ell \geq 3$, then $n = 2^{\alpha_1} - 2^{\alpha_2} + n''$ where
$n'' > 0$ and hence $n > 2^{\alpha_1} - 2^{\alpha_2} \geq 2^{\alpha_2}$.
So we have here that $2^{\alpha_2}<n<2^{\alpha_1}$ and therefore 
$\alpha_1 = \lceil \lg n\rceil$. Since $(\alpha_2,\ldots, \alpha_{\ell})$
determines an ABR of $n' = 2^{\alpha_1} - n$, the lemma follows by induction
on $\ell$.
\end{proof}
We now have a formula for $c(n)$ and $h(n)$ in terms of the ABR 
of $n$.
\begin{theorem}
\label{thm:k-formula}
For $n\in\nats$ let $n = 2^{\alpha_1} - 2^{\alpha_2} + \cdots + (-1)^{\ell-1}2^{\alpha_{\ell}}$
be the ABR of $n$. Then 
\[
c(n) = \left\{ \begin{array}{ll}
                 \ell   & \mbox{ if $n$ is odd }  (\alpha_{\ell} = 0), \\
                 \ell+1 & \mbox{ if $n$ is even } (\alpha_{\ell} \geq 1).
               \end{array}
       \right.
\]
Consequently, for the number of HCBP $h(n)$ we have 
\[
h(n) = \left\{ \begin{array}{ll}
                 \ell-1 & \mbox{ if $n$ is odd }  (\alpha_{\ell} = 0), \\
                 \ell   & \mbox{ if $n$ is even } (\alpha_{\ell} \geq 1).
               \end{array}
       \right.
\]
\end{theorem}
\begin{proof}
By (\ref{eqn:recc}) and Corollary~\ref{cor:recurr-h} the statement is clearly true 
for any power of two $n=2^{\alpha}$. The rest follows by 
Lemma~\ref{lmm:lgalt} and induction on $n$,  and
(\ref{eqn:recc}) and Corollary~\ref{cor:recurr-h}
\end{proof}
We can extract alternative recursions for $c(n)$ and $h(n)$ from
the above Theorem~\ref{thm:k-formula}, different from
the ones given in (\ref{eqn:recc}) and Corollary~\ref{cor:recurr-h}.

Note that for an even $n$, the ABR of $n/2$ is obtained
by subtracting 1 from each of the $\alpha_i$
in the ABR of $n$. If $n=4m$, then $n/2 = 2m$ is still even
and we have by Theorem~\ref{thm:k-formula} that $c(4m) = c(2m)$.
If $n = 4m+2$, then $n/2 = 2m+1$ is odd and we 
obtain by Theorem~\ref{thm:k-formula} that $c(4m+2) = c(2m+1)+1$.

For odd $n$, we must consider the following cases.
If $n = 8m+1$ (resp.~$n= 8m+7$), then the ABR of $4m+1$ 
(resp.~$4m+3$) is obtained by
subtracting 1 from all $\alpha_i$ except the last one $\alpha_{\ell}$
in the ABR of $8m+1$ (resp.~$8m+7$). As all are odd numbers, 
we we have by Theorem~\ref{thm:k-formula} that $c(8m+1) = c(4m+1)$
(resp.~$c(8m+7) = c(4m+3)$.) If $n = 8m+3$, then
the ABR of $4m+1$ is obtained by
subtracting 1 from all $\alpha_i$ where $i\in\{1,\ldots,\ell-2\}$
in the ABR of $8m+3$, and replacing the last two summands
$2^2 - 2^0$ for $i\in\{\ell-1,\ell\}$ by $2^0$.
As both $8m+3$ and $4m+1$ are odd, and the latter has
on fewer terms in its ABR, we have $c(8m+3) = c(4m+1)+1$. Finally,
if $n=8m+5$, then the ABR of $4m+3$ is obtained by
subtracting 1 from all $\alpha_i$ where $i\in\{1,\ldots,\ell-2\}$
in the ABR of $8m+5$, then removing $-2^2$ for $i=\ell-1$ and 
replacing $2^0$ by $-2^0$ for $i=\ell$.
Again, as both $8m+5$ and $4m+3$ are odd, and the latter has
on fewer terms in its ABR, we have $c(8m+5) = c(4m+3)+1$.
 -- As $h(n) = c(n) - 1$, we therefore have the following
alternative recursion for $h(n)$.
\begin{corollary}
\label{cor:rec-mod8}
For $n\in\nats$ the number $h(n)$ of HCBPs of $n$ is determined
by 
\[
(h(n))_{n=1}^8 = (0,1,1,1,2,2,1,1)
\] 
and the following recurrence
\[
h(n) = \left\{ \begin{array}{ll}
                 h(n/2)         & \mbox{ if }n\equiv 0 (\bmod 4), \\
                 h(n/2) + 1     & \mbox{ if }n\equiv 2 (\bmod 4), \\
                 h((n+1)/2)     & \mbox{ if }n\equiv 1 (\bmod 8), \\
                 h((n-1)/2) + 1 & \mbox{ if }n\equiv 3 (\bmod 8), \\
                 h((n+1)/2) + 1 & \mbox{ if }n\equiv 5 (\bmod 8), \\
                 h((n-1)/2)     & \mbox{ if }n\equiv 7 (\bmod 8).
               \end{array}
      \right.
\]
\end{corollary}
{\sc Remark:} There is a strong resemblance between $c(n)$ and $h(n)$ on
one hand, and $s(n)$, the number of digits in the binary representation
of $n$ from (\ref{eqn:s}), on the other; firstly $s(n)$ satisfies the recursion
$s(n) = s(n - 2^{\lfloor \lg n\rfloor}) + 1$ with $s(0) = 0$, a very
similar recursion to the one in (\ref{eqn:recc}). These recursions
are both obtained ``from the top'', or ``from the left'', in the sense that we consider
what happens when we remove the first power of 2 in the usual binary
representation of $n$ and in the ABR of $n$ respectively. On the other
hand the recursion in (\ref{eqn:s}) and Corollary~\ref{cor:rec-mod8}
are both obtained ``from behind'', or ``from the right'', by considering 
the removal of the last power
of 2 in the usual binary representation of $n$ and in the ABR of 
$n$ respectively.

\vspace{ 3 mm}

Note that by Observation~\ref{obs:lg-1} we have
that $2\leq c(n)\leq \lceil \lg n\rceil$ and 
$1\leq h(n)\leq \lceil \lg n\rceil-1$ for every $n\geq 3$.
We conclude this section by a sharpening of this about the
number of HCBPs of $n$ with $\lceil \lg n\rceil = k$ given.
\begin{proposition}
\label{prp:shar}
For $k\geq 2$ and $\ell\in\{1,\ldots,k-1\}$
let ${\cal H}_k(\ell) = \{n : \lceil \lg n\rceil = k \mbox{ and } h(n) = \ell \}$.
Then 
\[
\left| {\cal H}_k(\ell)\right| = 2\binom{k-2}{\ell-1}.
\]
In particular for $k = \lceil \lg n\rceil$ we have in the extreme cases that 
\begin{enumerate}
  \item $h(n) = 1$ iff $n \in \{2^k - 1, 2^k\}$,
  \item $h(n) = k-1$ iff $n \in \{(2^{k+1} + (-1)^k)/3, (2^{k+1} + (-1)^k)/3 +(-1)^k\}$.
\end{enumerate}
\end{proposition}
\begin{proof}
Assume that $\lceil \lg n\rceil = k$ and hence $2^{k-1} < n \leq 2^k$.
If $(\alpha_1,\ldots,\alpha_{i})$ determine the ABR of $n$, then
$\alpha_1 = k$ is determined. 

{\sc First case:} If $n$ is even, then $\alpha_{i} \geq 1$ and hence
the distinct decreasing exponents $\alpha_2, \cdots, \alpha_{i-1}, \alpha_{i}+1$ 
are determined by an $(i-1)$-subset of $\{2,\ldots,k-1\}$ of which
there are exactly $\binom{k-2}{i-1}$. By Theorem~\ref{thm:k-formula} 
the number of even $n\in {\cal H}_k(\ell)$ 
is therefore given by $\binom{k-2}{\ell-1}$.

{\sc Second case:} If $n$ is odd, then $\alpha_{i} = 0$ is determined
and hence the distinct decreasing exponents $\alpha_2, \cdots, \alpha_{i-1}$
are determined by an $(i-2)$-subset of $\{2,\ldots,k-1\}$ of which
there are exactly $\binom{k-2}{i-2}$. By Theorem~\ref{thm:k-formula} 
the number of odd $n\in {\cal H}_k(\ell)$ 
is therefore given by $\binom{k-2}{(\ell+1) - 2} = \binom{k-2}{\ell-1}$.
Hence $\left| {\cal H}_k(\ell)\right| = 2\binom{k-2}{\ell-1}$, which completes
the first part of the Proposition.

By Corollary~\ref{cor:recurr-h} we clearly have $h(2^k) = h(2^k-1) = 1$,
and by the first part there are precisely two numbers $n$ with
$\lceil \lg n\rceil = k$ with $h(n) = 1$. 

Finally, if $n = a_k = (2^{k+1}+(-1)^k)/3 = 
2^k - 2^{k-1} + \cdots + (-1)^{k-2}2^2 + (-1)^{k-1}2^0$,
then we have seen right above Observation~\ref{obs:lg-1} that $h(n) = c(n) - 1 = k-1$.
Also, for an even $n  = b_k = a_k + (-1)^k = 
2^k - 2^{k-1} + \cdots + (-1)^{k-3}2^3 + (-1)^{k-2}2^2$, 
then we have by Corollary~\ref{cor:recurr-h}
that $h(b_k) = h(b_{k-1}) + 1 = h(b_2) + k-2 = h(4) + k-2 = k-1$.
Since by the first part there are at most two numbers $n$ with 
$\lceil \lg n\rceil = k$ and $h(n) = k-1$, this completes the proof.
\end{proof}

\subsection*{Acknowledgments}  

Sincere thanks to the anonymous referees for ...

\flushright{\today}

\end{document}